\documentclass[a4paper, 12pt]{amsart}
\usepackage{amsmath, amssymb, amsthm, microtype}
\usepackage[bottom=1in, left=1in, right=1in, top=1in]{geometry}
\usepackage{color}

\makeatletter
\renewcommand{\pod}[1]{\mathchoice
  {\allowbreak \if@display \mkern 18mu\else \mkern 8mu\fi (#1)}
  {\allowbreak \if@display \mkern 18mu\else \mkern 8mu\fi (#1)}
  {\mkern4mu(#1)}
  {\mkern4mu(#1)}
}

\newtheorem{theorem}{Theorem}[section]

\newtheorem{lemma}[theorem]{Lemma}

\newtheorem{proposition}[theorem]{Proposition}

\theoremstyle{remark}

\theoremstyle{definition}
\newtheorem{definition}[theorem]{Definition}

\newcommand{\cA}{\mathcal{A}}

\newcommand{\cC}{\mathcal{C}}

\newcommand{\cM}{\mathcal{M}}
\newcommand{\cP}{\mathcal{P}}
\newcommand{\cX}{\mathcal{X}}
\newcommand{\Li}{\operatorname{Li}}

\newcommand{\bN}{\mathbb{N}}
\newcommand{\bF}{\mathbb{F}}
\newcommand{\bQ}{\mathbb{Q}}
\newcommand{\bZ}{\mathbb{Z}}
\newcommand{\mfa}{\mathfrak{a}}
\newcommand{\mfi}{\mathfrak{i}}
\newcommand{\mfP}{\mathfrak{P}}
\newcommand{\mfp}{\mathfrak{p}}
\newcommand{\mfq}{\mathfrak{q}}
\newcommand{\mfs}{\mathfrak{s}}
\newcommand{\mfu}{\mathfrak{u}}
\newcommand{\mfv}{\mathfrak{v}}
\newcommand{\efp}{\#E(\bF_p)}
\newcommand{\ol}{\overline}

\AtBeginDocument{%
   \def\MR#1{}
}

\title[Extreme values of $\omega(\efp)$]{Orders of reductions of elliptic curves with many and few prime factors}

\begin{document}

\author{Lee Troupe}
\address{Department of Mathematics, Boyd Graduate Studies Research Center, University of Georgia, Athens, GA 30602, USA}
\email{ltroupe@math.uga.edu}
\let\thefootnote\relax\footnote{The author was partially supported by NSF RTG Grant DMS-1344994.}

\begin{abstract}
In this paper, we investigate extreme values of $\omega(\efp)$, where $E/\bQ$ is an elliptic curve with complex multiplication and $\omega$ is the number-of-distinct-prime-divisors function. For fixed $\gamma > 1$, we prove that
\[
\#\{p \leq x : \omega(\efp) > \gamma\log\log x\} = \frac{x}{(\log x)^{2 + \gamma\log\gamma - \gamma + o(1)}}.
\]
The same result holds for the quantity $\#\{p \leq x : \omega(\efp) < \gamma\log\log x\}$ when $0 < \gamma < 1$. The argument is worked out in detail for the curve $E : y^2 = x^3 - x$, and we discuss how the method can be adapted for other CM elliptic curves.
\end{abstract}

\maketitle

\section{Introduction}


Let $E/\bQ$ be an elliptic curve. For primes $p$ of good reduction, one has
\[
E(\bF_p) \simeq \bZ/d_p\bZ \oplus \bZ/e_p\bZ
\]
where $d_p$ and $e_p$ are uniquely determined natural numbers such that $d_p$ divides $e_p$. Thus, $\efp = d_pe_p$. We concern ourselves with the behavior $\omega(\efp)$, where $\omega(n)$ denotes the number of distinct prime factors of the number $n$, as $p$ varies over primes of good reduction. Work has been done already in this arena: If the curve $E$ has CM, Cojocaru \cite[Corollary 6]{coj05} showed that the normal order of $\omega(\efp)$ is $\log\log p$, and a year later, Liu \cite{liu06} established an elliptic curve analogue of the celebrated Erd{\H o}s - Kac theorem: For any elliptic curve $E/\bQ$ with CM, the quantity
\[
\frac{\omega(\efp) - \log\log p}{\sqrt{\log\log p}}
\]
has a Gaussian normal distribution. In particular, $\omega(\efp)$ has normal order $\log \log p$ and standard deviation $\sqrt{\log\log p}$. (These results hold for elliptic curves without CM, if one assumes GRH.)

In light of the Erd{\H o}s - Kac theorem, one may ask how often $\omega(n)$ takes on extreme values, e.g. values greater than $\gamma \log \log n$, for some fixed $\gamma > 1$. A more precise version of the following result appears in \cite{Erd1978-1979}; its proof is due to Delange.

\begin{theorem}
Fix $\gamma > 1$. As $x \to \infty$,
\[
\#\{n \leq x : \omega(n) > \gamma\log\log x\} = \frac{x}{(\log x)^{1 + \gamma\log\gamma - \gamma + o(1)}}.
\]
\end{theorem}

Presently, we establish an analogous theorem for the quantity $\omega(\efp)$, where $E/\bQ$ is an elliptic curve with CM.

\begin{theorem}
Let $E/\bQ$ be an elliptic curve with CM. For $\gamma > 1$ fixed,
\[
\#\{p \leq x : \omega(\efp) > \gamma\log\log x\} = \frac{x}{(\log x)^{2 + \gamma\log\gamma - \gamma + o(1)}}.
\]

\noindent The same result holds for the quantity $\#\{p \leq x : \omega(\efp) < \gamma\log\log x\}$ when $0 < \gamma < 1$.
\end{theorem}

In what follows, the above theorem will be proved for $E/\bQ$ with $E : y^2 = x^3 - x$. Essentially the same method can be used for any elliptic curve with CM; refer to the discussion in \S 4 of \cite{poltitec}. To establish the theorem, we prove corresponding upper and lower bounds in sections \S 3 and \S 4, respectively.

\bigskip

\noindent\textit{Remark.} One can ask similar questions about other arithmetic functions applied to $\efp$. For example, Pollack has shown \cite{poltitec} that, if $E$ has CM, then
\[
\sideset{}{'}\sum_{p \leq x} \tau(\efp) \sim c_E \cdot x,
\]
where the sum is restricted to primes $p$ of good ordinary reduction for $E$. Several elements of Pollack's method of proof will appear later in this manuscript.

\bigskip

\noindent\textbf{Notation.} $K$ will denote an extension of $\bQ$ with ring of integers $\bZ_K$. For each ideal $\mathfrak{a} \subset \bZ_K$, we write $\Vert \mfa \Vert$ for the norm of $\mfa$ (that is, $\Vert \mfa \Vert = \#\bZ_K/\mfa$) and $\Phi(\mfa) = \#(\bZ_K/\mfa)^{\times}$. The function $\omega$ applied to an ideal $\mfa \subset \bZ_K$ will denote the number of distinct prime ideals appearing in the factorization of $\mfa$ into a product of prime ideals. For $\alpha \in \bZ_K$, $\Vert \alpha \Vert$ and $\Phi(\alpha)$ denote those functions evaluated at the ideal $(\alpha)$. If $\alpha$ is invertible modulo an ideal $\mfu \subset \bZ_K$, we write $\gcd(\alpha, \mfu) = 1$. The notation $\log_kx$ will be used to denote the $k$th iterate of the natural logarithm; this is not to be confused with the base-$k$ logarithm. The letters $p$ and $q$ will be reserved for rational prime numbers. We make frequent use of the notation $\ll, \gg$ and $O$-notation, which has its usual meaning. Other notation may be defined as necessary.

\noindent\textbf{Acknowledgements.} The author thanks Paul Pollack for a careful reading of this manuscript and many helpful suggestions.

\section{Useful propositions}

One of our primary tools will be a version of Brun's sieve in number fields. The following theorem can be proved in much the same way that one obtains Brun's pure sieve in the rational integers, cf. \cite[\S 6.4]{polnabd}.

\begin{theorem}\label{brun}
Let $K$ be a number field with ring of integers $\bZ_K$. Let $\cA$ be a finite sequence of elements of $\bZ_K$, and let $\cP$ be a finite set of prime ideals. Define 
\[
S(\cA, \cP) := \#\{a \in \cA : \gcd(a, \mfP) = 1\},\text{     where   } \mfP := \prod_{\mfp \in \cP} \mfp.
\]
For an ideal $\mfu \subset \bZ_K$, write $A_\mfu := \#\{a \in \cA : a \equiv 0 \pmod \mfu\}$. Let $X$ denote an approximation to the size of $\cA$. Suppose $\delta$ is a multiplicative function taking values in $[0, 1]$, and define a function $r(\mfu)$ such that
\[
A_\mfu = X\delta(\mfu) + r(\mfu)
\]
for each $\mfu$ dividing $\mfP$. Then, for every even $m \in \bZ^{+}$,
\[
S(\cA, \cP) = X\prod_{\mfp \in \cP} (1 - \delta(\mfp)) + O\bigg(\sum_{\mfu \mid \mfP, \, \omega(\mfu) \leq m} |r(\mfu)|\bigg) + O\bigg(X \sum_{\mfu \mid \mfP, \, \omega(\mfu) \geq m} \delta(\mfu)\bigg).
\]
All implied constants are absolute.
\end{theorem}

In our estimation of $O$-terms arising from the use of Proposition \ref{brun}, we will make frequent use of the following analogue of the Bombieri-Vinogradov theorem, which we state for an arbitrary imaginary quadratic field $K/\bQ$ with class number 1. For $\alpha \in \bZ_K$ and an ideal $\mfq \subset \bZ_K$, write
\[
\pi(x; \mfq, \alpha) = \#\{\mu \in \bZ_K : \Vert \mu \Vert \leq x, \mu \equiv \alpha \pmod \mfq\}.
\]

\begin{proposition}\label{bv}
For every $A > 0$, there is a $B > 0$ so that
\[
\sum_{\Vert \mfq \Vert \leq x^{1/2}(\log x)^{-B}} \max_{\alpha: \gcd(\alpha, \mfu) = 1} \max_{y \leq x} \vert \pi(y; \mfq, \alpha) - w_K \cdot \frac{\Li(y)}{\Phi(\mfq)} \vert \ll \frac{x}{(\log x)^A},
\]
where the above sum and maximum are taken over $\mfq \subset \bZ_K$ and $\alpha \in \bZ_K$. Here $w_K$ denotes the size of the group of units of $\bZ_K$
\end{proposition}

The above follows from Huxley's analogue of the Bombieri-Vinogradov theorem for number fields \cite{hux71}; see the discussion in \cite[Lemma 2.3]{poltitec}.

The following proposition is an analogue of Mertens' theorem for imaginary quadratic fields. It follows immediately from Theorem 2 of \cite{ros99}.

\begin{proposition}\label{mertens}
Let $K/\bQ$ be an imaginary quadratic field and let $\alpha_K$ denote the residue of the associated Dedekind zeta function, $\zeta_K(s)$, at $s = 1$. Then
\[
\prod_{\Vert \mfp \Vert \leq x} \Big( 1 - \frac{1}{\Vert \mfp \Vert} \Big)^{-1} \sim e^\gamma \alpha_K \log x,
\]
where the product is over all prime ideals $\mfp$ in $\bZ_K$. Here (and only here), $\gamma$ is the Euler-Mascheroni constant.
\end{proposition}

Note also that the ``additive version'' of Mertens' theorem, i.e.,
\[
\sum_{\Vert \mfp \Vert \leq x} \frac{1}{\Vert \mfp \Vert} = \log_2 x + B_K + O_K\bigg(\frac{1}{\log x}\bigg)
\]
for some constant $B_K$, holds in this case as well; it appears as Lemma 2.4 in [Rosen].

Finally, we will make use of the following estimate for elementary symmetric functions \cite[p. 147, Lemma 13]{halroth83}.

\begin{lemma}\label{elementary}
Let $y_1, y_2, \ldots, y_M$ be $M$ non-negative real numbers. For each positive integer $d$ not exceeding $M$, let
\[
\sigma_d = \sum_{1 \leq k_1 < k_2 < \cdots < k_d \leq M} y_{k_1}y_{k_2}\cdots y_{k_d},
\]
so that $\sigma_d$ is the $d$th elementary symmetric function of the $y_k$'s. Then, for each $d$, we have
\[
\sigma_d \geq \frac{1}{d!} \sigma_1^d\Bigg(1 - \binom{d}{2}\frac{1}{\sigma_1^2}\sum_{k = 1}^M y_k^2\Bigg).
\]
\end{lemma}

%

\section{An upper bound}

\begin{theorem}\label{upperbound}
Let $E$ be the elliptic curve $E : y^2 = x^3 - x$ and fix $\gamma > 1$. Then
\[
\#\{p \leq x : \omega(\efp) > \gamma \log_2 x\} \ll_\gamma \frac{x(\log_2 x)^5}{(\log x)^{2 + \gamma\log\gamma - \gamma}}.
\]
The same statement is true if instead $0 < \gamma < 1$ and the strict inequality is reversed on the left-hand side.
\end{theorem}

Before proving Theorem \ref{upperbound}, we refer to \cite[Table 2]{ju08} for the following useful fact concerning the numbers $\efp$: For primes $p \leq x$ with $p \equiv 1 \pmod 4$, we have
\begin{align}\label{efpnorm}
\efp = p + 1 - (\pi + \ol{\pi}) = (\pi - 1)\ol{(\pi - 1)},
\end{align}
where $\pi \in \bZ[i]$ is chosen so that $p = \pi\ol{\pi}$ and $\pi \equiv 1 \pmod{(1 + i)^3}$. (Such $\pi$ are sometimes called \emph{primary}.) This determines $\pi$ completely up to conjugation.

We begin the proof of Theorem \ref{upperbound} with the following lemma, which will allow us to disregard certain problematic primes $p$.

\begin{lemma}\label{discard}
Let $x \geq 3$ and let $P(n)$ denote the largest prime factor of $n$. Let $\cX$ denote the set of $n \leq x$ for which either of the following properties fail: 
\begin{itemize}
	\item[(i)] $P(n) > x^{1/6\log_2 x}$
	\item[(ii)] $P(n)^2 \nmid n$.
\end{itemize}
Then, for any $A > 0$, the size of $\cX$ is $O(x/(\log x)^A)$.
\end{lemma}

The following upper bound estimate of de Bruijn \cite[Theorem 2]{db66} will be useful in proving the above lemma.

\begin{proposition}\label{smooth}
Let $x \geq y \geq 2$ satisfy $(\log x)^2 \leq y \leq x$. Whenever $u := \frac{\log x}{\log y} \to \infty$, we have
\[
\Psi(x, y) \leq x/u^{u + o(u)}.
\]
\end{proposition}

\begin{proof}[Proof of Lemma \ref{discard}.]
If $n \in \cX$, then either (a) $P(n) \leq x^{1/6\log_2 x}$ or (b) $P(n) > x^{1/6\log_2 x}$ and $P(n)^2 \mid n$. By Proposition \ref{smooth}, the number of $n \leq x$ for which (a) holds is $O(x/(\log x)^A)$ for any $A > 0$, noting that $(\log x)^A \ll (\log x)^{\log_3 x} = (\log_2 x)^{\log_2 x}$. The number of $n \leq x$ for which (b) holds is 
\[
\ll x\sum_{p > x^{1/6\log_2 x}} p^{-2} \ll x\exp(-\log x/6 \log_2x),
\]
and this is also $O(x/(\log x)^A)$.
\end{proof}

We would like to use Lemma \ref{discard} to say that a negligible amount of the numbers $\efp$, for $p \leq x$, belong to $\cX$. The following lemma allows us to do so.

\begin{lemma}\label{efpsafe}
The number of $p \leq x$ with $\efp \in \cX$ is $O(x/(\log x)^B)$, for any $B > 0$.
\end{lemma}

\begin{proof}
Suppose $\efp = b \in \cX$. Then, by (\ref{efpnorm}), $b = \Vert \pi - 1 \Vert$, where $\pi \in \bZ[i]$ is a Gaussian prime lying above $p$. Thus, the number of $p \leq x$ with $\efp = b$ is bounded from above by the number of Gaussian integers with norm $b$, which, by \cite[Theorem 278]{hw00}, is $4\sum_{d \mid b} \chi(d)$, where $\chi$ is the nontrivial character modulo 4. Now, using the Cauchy-Schwarz inequality and Lemma \ref{discard},
\begin{align*}
4 \sum_{b \in \cX} \sum_{d \mid b} \chi(d) \leq 4 \sum_{b \in \cX} \tau(b) &\leq 4\Big(\sum_{b \in \cX} 1\Big)^{1/2} \Big(\sum_{b \in \cX} \tau(b)^2 \Big)^{1/2} \\
 &\ll \Big(\frac{x}{(\log x)^{A}}\Big)^{1/2} \Big( x\log^3 x \Big)^{1/2} = \frac{x}{(\log x)^{A/2 - 3/2}}.
\end{align*}
Since $A > 0$ can be chosen arbitrarily, this completes the proof.
\end{proof}

For $k$ a nonnegative integer, define $N_k$ to be the number of primes $p \leq x$ of good ordinary reduction for $E$ such that $\efp$ possesses properties $(i)$ and $(ii)$ from the above lemma and such that $\omega(\efp) = k$. Then, in the case when $\gamma > 1$,
\[
\#\{p \leq x : \omega(\efp) > \gamma \log\log x\} = \sum_{k > \gamma \log_2 x} N_k + O\Big(\frac{x}{(\log x)^A}\Big)
\]
for any $A > 0$. Our task is now to bound $N_k$ from above in terms of $k$. Evaluating the sum on $k$ then produces the desired upper bound.

It is clear that
\begin{align}\label{nkub}
N_k \leq \sum_{\substack{a \leq x^{1 - 1/6\log_2 x} \\ \omega(a) = k-1}} \sum_{\substack{p \leq x \\ p \equiv 1 \pmod 4 \\a \mid \efp \\ \efp/a \text{ prime}}} 1.
\end{align}

To handle the inner sum, we need information on the integer divisors of $\efp$, where $p \leq x$ and $p \equiv 1 \pmod 4$. We employ the analysis of Pollack in his proof of \cite[Theorem 1.1]{poltitec}, which we restate here for completeness.

By (\ref{efpnorm}), we have $a \mid \efp$ if and only if $a \mid (\pi - 1)\ol{(\pi - 1)} = \Vert \pi - 1\Vert$. With this in mind, we have
\[
\sum_{\substack{a \leq x^{1 - 1/6\log\log x} \\ \omega(a) = k-1}} \sum_{\substack{p \leq x \\ p \equiv 1 \pmod 4 \\a \mid \efp \\ \efp/a \text{ prime}}} 1 = \frac{1}{2} \sum_{\substack{a \leq x^{1 - 1/6\log\log x} \\ \omega(a) = k-1}} \sideset{}{'}\sum_{\substack{\pi \, : \, \Vert\pi\Vert \leq x \\ \pi \equiv 1 \pmod{(1 + i)^3} \\ a \mid \Vert\pi - 1\Vert \\ \Vert\pi - 1\Vert/a \text{ prime}}} 1,
\]
where the $'$ on the sum indicates a restriction to primes $\pi$ lying over rational primes $p \equiv 1 \pmod 4$.

\subsection{Divisors of shifted Gaussian primes.} The conditions on the primed sum above can be reformulated purely in terms of Gaussian integers.

\begin{definition}\label{setsa} For a given integer $a \in \bN$, write $a = \prod_q q^{v_q}$, with each $q$ prime. For each $q \mid a$ with $q \equiv 1 \pmod 4$, write $q = \pi_q\ol{\pi}_q$. Define a set $S_a$ which consists of all products $\alpha$ of the form
\[
\alpha = (1 + i)^{v_2} \prod_{\substack{q \mid a \\ q \equiv 3 \pmod 4}} q^{\lceil v_q/2 \rceil} \prod_{\substack{q \mid a \\ q \equiv 1 \pmod 4}} \alpha_q,
\]
where $\alpha_q \in \{\pi_q^i \ol{\pi}_q^{v_q - i} : i = 0, 1, \ldots, v_q\}$.
\end{definition}

Notice that the condition $a \mid \Vert \pi - 1 \Vert$ is equivalent to $\pi - 1$ being divisible by some element of the set $S_a$. We can therefore write
\begin{align}\label{alphasum}
\sum_{\substack{a \leq x^{1 - 1/6\log\log x} \\ \omega(a) = k-1}} \sum_{\substack{p \leq x \\ p \equiv 1 \pmod 4 \\a \mid \efp \\ \efp/a \text{ prime}}} 1 \leq \frac{1}{2}\sum_{\substack{a \leq x^{1 - 1/6\log\log x} \\ \omega(a) = k-1}} \sum_{\alpha \in S_a} \sideset{}{'}\sum_{\substack{\pi \, : \, \Vert\pi\Vert \leq x \\ \pi \equiv 1 \pmod{(1 + i)^3} \\ \alpha \mid \pi - 1 \\ \Vert \pi - 1 \Vert /a  \text{ prime}}} 1.
\end{align}

Now, for any $\alpha \in S_a$, we have
\[
\alpha\ol{\alpha} = a\prod_{q \equiv 3 \pmod 4} q^{2\lceil v_q/2 \rceil - v_q}.
\]
Observe that
\begin{align*}
\frac{\Vert \pi - 1 \Vert}{a} = \frac{(\pi-1)(\ol{\pi - 1})}{\alpha\ol{\alpha}}\prod_{q \equiv 3 \pmod 4} q^{2\lceil v_q/2 \rceil - v_q}.
\end{align*}
Therefore, if $\frac{\Vert \pi - 1 \Vert}{a}$ is to be prime, the number $a$ must satisfy exactly one of the following properties:
\begin{itemize}
\item[1.] The number $a$ is divisible by exactly one prime $q \equiv 3 \pmod 4$ with $v_q$ an odd number, and $\alpha = u(\pi - 1)$ where $u \in \bZ[i]$ is a unit; or
\item[2.] All primes $q \equiv 3 \pmod 4$ which divide $a$ have $v_q$ even, and $(\pi - 1) / \alpha$ is a prime in $\bZ[i]$.
\end{itemize}

This splits the outer sum in (\ref{alphasum}) into two components.

\begin{lemma}\label{ubcase1}
We have
\[
\sideset{}{^\flat}\sum_{\substack{a \leq x^{1 - 1/6\log\log x} \\ \omega(a) = k-1}} \sum_{\alpha \in S_a} \sideset{}{'}\sum_{\substack{\pi \, : \, \Vert\pi\Vert \leq x \\ \pi \equiv 1 \pmod{(1 + i)^3} \\ (\pi - 1)/\alpha \in U}} 1 = O\bigg(\frac{x}{\log^A x}\bigg),
\]
where $U$ is the set of units in $\bZ[i]$ and the $\flat$ on the outer sum indicates a restriction to integers $a$ such that there is a unique prime power $q^{v_q} \Vert a$ with $q \equiv 3 \pmod 4$ and $v_q$ odd.
\end{lemma}

\begin{proof}
If $\alpha = u(\pi - 1)$ for $u \in U$, then there are at most four choices for $\pi$, given $\alpha$. Thus
\[
\sideset{}{^\flat}\sum_{\substack{a \leq x^{1 - 1/6\log\log x} \\ \omega(a) = k-1}} \sum_{\alpha \in S_a} \sideset{}{'}\sum_{\substack{\pi \, : \, \Vert\pi\Vert \leq x \\ \pi \equiv 1 \pmod{(1 + i)^3} \\ \alpha = u(\pi - 1)}} 1 \leq 4 \sideset{}{^\flat}\sum_{\substack{a \leq x^{1 - 1/6\log\log x} \\ \omega(a) = k-1}} |S_a|.
\]
We have $|S_a| = \prod_{q \equiv 1 \pmod 4} (v_q + 1)$; this is bounded from above by the divisor function on $a$, which we denote $\tau(a)$. Therefore, the above is
\[
\ll \sum_{a \leq x^{1 - 1/6\log\log x}} \tau(a) \ll x^{1 - 1/6\log_2 x}(\log x),
\]
which is $O(x/\log^Ax)$ for any $A > 0$.
\end{proof}

The second case provides the main contribution to the sum.

\begin{lemma}\label{ubcase2}
Let $a \leq x^{1 - 1/6\log\log x}$ with $\omega(a) = k-1$ such that all primes $q \equiv 3 \pmod 4$ dividing $a$ have $v_q$ even. Let $\alpha \in S_a$. Then
\[
\sideset{}{'}\sum_{\substack{\pi \, : \, \Vert\pi\Vert \leq x \\ \pi \equiv 1 \pmod{(1 + i)^3} \\ \alpha \mid \pi - 1 \\ (\pi - 1)/\alpha \text{ prime}}} 1 \ll \frac{x(\log_2 x)^5}{\Vert \alpha \Vert(\log x)^2}
\]
uniformly over all $a$ as above and $\alpha \in S_a$.
\end{lemma}

\begin{proof}

If $\pi \equiv 1 \pmod \alpha$, then $\pi = 1 + \alpha\beta$ for some $\beta \subset \bZ[i]$. Thus $\beta = \frac{\pi - 1}{\alpha}$, and so $\Vert \beta \Vert \leq \frac{2x}{\Vert \alpha \Vert}$. Let $\cA$ denote the sequence of elements in $\bZ[i]$ given by
\[
\Big\{ \beta(1 + \alpha\beta) : \Vert \beta \Vert \leq \frac{2x}{\Vert \alpha \Vert} \Big\}.
\]
Define $\cP = \{\mfp \subset \bZ[i] : \Vert \mfp \Vert \leq z\}$ where $z$ is a parameter to be chosen later. Then, in the notation of Theorem \ref{brun},
\[
\sideset{}{'}\sum_{\substack{\pi \, : \, \Vert\pi\Vert \leq x \\ \pi \equiv 1 \pmod{(1 + i)^3} \\ \alpha \mid \pi - 1 \\ (\pi - 1)/\alpha \text{ prime}}} 1 \leq S(\cA, \cP) + O(z).
\]
Here, the $O(z)$ term comes from those $\pi \in \bZ[i]$ such that both $\pi$ and $(\pi - 1)/\alpha$ are primes of norm less than $z$.

For $\mfu \subset \bZ[i]$, write $A_\mfu = \#\{a \in \cA : a \equiv 0 \pmod \mfu\}$. An element $\mfa \in \cA$ is counted by $A_\mfu$ if and only if a generator of $\mfu$ divides $a$. Thus, by familiar estimates on the number of integer lattice points contained in a circle, $A_\mfu$ satisfies the equation
\[
A_\mfu = \frac{2\pi x}{\Vert\alpha\Vert} \frac{\nu(\mfu)}{\Vert \mfu \Vert} + O\Big(\nu(\mfu)\frac{\sqrt{x}}{(\Vert \alpha\Vert\Vert\mfu \Vert)^{1/2}}\Big),
\]
where 
\[
\nu(\mfu) = \#\{\beta \pmod \mfu : \beta(1 + \alpha\beta) \equiv 0 \pmod \mfu\}.
\]
We apply Theorem \ref{brun} with
\[
X = \frac{2\pi x}{\Vert\alpha\Vert} \,\,\,\,\,\, \text{and} \,\,\,\,\,\, \delta(\mfu) = \frac{\nu(\mfu)}{\Vert\mfu\Vert}.
\]
With these choices, we have
\[
r(\mfu) = O\Big(\nu(\mfu)\frac{\sqrt{x}}{(\Vert\alpha\Vert\Vert\mfu\Vert)^{1/2}|}\Big).
\]
Then, for any even integer $m \geq 0$,
\begin{align}\label{sap1}
S(\cA, \cP) = \frac{2\pi x}{\Vert\alpha\Vert} \prod_{\Vert \mfp \Vert \leq z}\bigg(1 - &\frac{\nu(\mfp)}{\Vert\mfp\Vert}\bigg) + O\bigg(\frac{\sqrt{x}}{\Vert\alpha\Vert^{1/2}}\sum_{\substack{\mfu \mid \mfP \\ \omega(\mfu) \leq m}} \frac{\nu(\mfu)}{\Vert\mfu\Vert^{1/2}}\bigg) \\ &+ O\bigg(\frac{x}{\Vert\alpha\Vert}\sum_{\substack{\mfu \mid \mfP \\ \omega(\mfu) \geq m}} \delta(\mfu)\bigg), \nonumber
\end{align}
where $\mfP = \prod_{\mfp \in \cP} \mfp$.

For a prime $\mfp$, we have $\nu(\mfp) = 2$ if $\alpha \not\equiv 0 \pmod \mfp$ and $\nu(\mfp) = 1$ otherwise. Therefore, the product in the first term is
\begin{align*}
\prod_{\substack{\Vert \mfp \Vert \leq z \\ \mfp \nmid (\alpha)}}\bigg(1 - \frac{2}{\Vert\mfp\Vert}\bigg) &\prod_{\substack{\Vert \mfp \Vert \leq z \\ \mfp \mid (\alpha)}}\bigg(1 - \frac{1}{\Vert\mfp\Vert}\bigg) \\ &\leq \prod_{\Vert \mfp \Vert \leq z}\bigg(1 - \frac{1}{\Vert\mfp\Vert}\bigg)^2 \prod_{\substack{\Vert \mfp \Vert \leq z \\ \mfp \mid (\alpha)}}\bigg(1 - \frac{1}{\Vert\mfp\Vert}\bigg)^{-1} \ll \frac{1}{(\log z)^2}\frac{\Vert\alpha\Vert}{\Phi(\alpha)},
\end{align*}
where in the last step we used Proposition \ref{mertens}.

Choose $z = x^{\frac{1}{200(\log_2 x)^2}}$. Then our first term in (\ref{sap1}) is
\[
\ll \frac{x(\log_2 x)^4}{\Phi(\alpha)(\log x)^2}.
\]
Recall that $\Vert \alpha \Vert = a$, and $a \leq x^{1 - 1/6\log_2 x}$. Since $\Phi(\alpha) \gg \Vert \alpha \Vert/\log_2 x$ (analogous to the minimal order for the usual Euler function, c.f. \cite[Theorem 328]{hw00}), the above is
\begin{align*}\label{maintermalpha}
\ll \frac{x(\log_2 x)^5}{\Vert \alpha \Vert(\log x)^2}.
\end{align*}

We now show that this ``main'' term dominates the two $O$-terms uniformly for $\alpha \in S_a$ and $a \leq x^{1 - 1/6\log_2 x}$. For the first $O$-term, we begin by noting that $\nu(\mfu)/\Vert \mfu \Vert^{1/2} \ll 1$. Then, taking $m = 10\lfloor \log_2 x \rfloor$, we have
\[
\sum_{\substack{\mfu \mid \mfP \\ \omega(\mfu) \leq m}} \frac{\nu(\mfu)}{\Vert \mfu \Vert^{1/2}} \ll \sum_{k = 0}^m \binom{\pi_K(z)}{k} \leq \sum_{k = 0}^m \pi_K(z)^k \leq 2\pi_K(z)^m \leq x^{1/20\log_2 x},
\]
where $\pi_K(z)$ denotes the number of prime ideals $\mfp \subset \bZ[i]$ with norm up to $z$. Therefore, the inequality
\[
\frac{x(\log_2 x)^5}{\Vert \alpha \Vert(\log x)^2} \gg \frac{x^{1/2 + 1/20\log_2 x}}{\Vert \alpha \Vert^{1/2}}
\]
holds for all $\alpha$ with $\Vert \alpha \Vert \leq x^{1 - 1/6\log_2x}$, as desired.

Next we handle the second $O$-term. The sum in this term is
\[
\sum_{\substack{\mfu \mid \mfP \\ \omega(\mfu) \geq m}} \delta(\mfu) \leq \sum_{s \geq m} \frac{1}{s!}\Big(\sum_{\substack{\Vert \mfp \Vert \leq z}} \frac{\nu(\mfp)}{\Vert \mfp \Vert} \Big)^s.
\]
Observe that, by Proposition \ref{mertens}, we have
\[
\sum_{\substack{\Vert \mfp \Vert \leq z}} \frac{\nu(\mfp)}{\Vert \mfp \Vert} \leq 2\log_2x + O(1).
\]
Thus, by the ratio test, one sees that the sum on $s$ is 
\[
\ll \frac{1}{m!}(2\log_2x + O(1))^m.
\]
Using Proposition \ref{mertens} followed by Stirling's formula, we obtain that the above quantity is
\begin{align*}
\frac{1}{m!}(2\log_2 x + O(1))^m &\leq \Big(\frac{2e\log_2 x + O(1)}{10\lfloor \log_2 x \rfloor}\Big)^{10\lfloor \log_2 x \rfloor} \\
&\ll \Big(\frac{e}{5}\Big)^{9\log_2 x} \leq \frac{1}{(\log x)^{5}}.
\end{align*}
So the second $O$-term is
\[
\ll \frac{x}{\Vert \alpha \Vert (\log x)^{5}},
\]
and this is certainly dominated by the main term.
\end{proof}

From Lemmas \ref{ubcase1} and \ref{ubcase2}, we see (\ref{nkub}) can be rewritten
\[
N_k \ll \frac{x(\log_2 x)^5}{(\log x)^2} \sum_{\substack{a \leq x^{1 - 1/6\log_2 x} \\ \omega(a) = k-1}} \frac{|S_a|}{a} + O\Big(\frac{x}{\log^A x}\Big),
\]
noting that $\Vert \alpha \Vert = a$ for all $a$ under consideration and all $\alpha \in S_a$. We are now in a position to bound $N_k$ from above in terms of $k$.

\begin{lemma}
We have
\[
\sum_{\substack{a \leq x^{1 - 1/6\log_2 x} \\ \omega(a) = k-1}} \frac{|S_a|}{a} \leq \frac{(\log_2 x + O(1))^{k-1}}{(k - 1)!}.
\]
\end{lemma}

\begin{proof} 
We have already seen that the size of $S_a$ is $\prod_{p \mid a : p \equiv 1 \pmod 4} (v_p + 1)$, where $v_p$ is defined by $p^{v_p} \parallel a$. Recall that in the current case, each prime $p \equiv 3 \pmod 4$ dividing $a$ appears to an even power. Therefore, we have
\begin{align}\label{sumona}
\sum_{\substack{a \leq x \\ \omega(a) = k-1}} \frac{|S_a|}{a} \leq \frac{1}{(k-1)!}\Bigg(\sum_{\substack{p^\ell \leq x \\ p \not\equiv 3 \pmod 4}} \frac{|S_{p^\ell}|}{p^\ell} + \sum_{\substack{p^{2k} \leq x \\ p \equiv 3 \pmod 4}} \frac{|S_{p^{2k}}|}{p^{2k}} + O(1)\Bigg)^{k-1}.
\end{align}
Note that $|S_{p^{2k}}| = 1$ for each prime $p \equiv 3 \pmod 4$. Thus we can absorb the sum corresponding to these primes into the $O(1)$ term, giving
\begin{align}\label{sumonamultinomial}
\sum_{\substack{a \leq x \\ \omega(a) = k-1}} \frac{|S_a|}{a} \ll \frac{1}{(k-1)!}\Bigg(\displaystyle\sum_{\substack{p^\ell \leq x \\ p \not\equiv 3 \pmod 4}} \frac{|S_{p^\ell}|}{p^\ell} + O(1)\Bigg)^{k-1}.
\end{align}
Now
\begin{align*}
\sum_{\substack{p^\ell \leq x \\ p \not\equiv 3 \pmod 4}} \frac{|S_{p^\ell}|}{p^\ell} &= \sum_{\substack{p^\ell \leq x \\ p \equiv 1 \pmod 4}} \frac{\ell + 1}{p^\ell} + O(1) \\
&= \sum_{\substack{p \leq x \\ p \equiv 1 \pmod 4}} \frac{2}{p} + O(1) \\
&= \log_2 x + O(1).
\end{align*}
Inserting this expression into (\ref{sumonamultinomial}) proves the lemma.
\end{proof}

\subsection{Finishing the upper bound.} We have shown so far that
\[
N_k \ll \frac{x(\log_2 x)^5}{(\log x)^2} \cdot \frac{(\log_2 x + O(1))^{k-1}}{(k-1)!}.
\]

We now sum on $k > \gamma\log_2 x$ for fixed $\gamma > 1$ to complete the proof of Theorem \ref{upperbound}. (The statement corresponding to $0 < \gamma < 1$ may be proved in a completely similar way.) Again using the ratio test and Stirling's formula, we have
\begin{align*}
\sum_{k > \gamma\log_2 x} &\frac{(\log_2 x + O(1))^{k-1}}{(k-1)!} \ll \bigg(\frac{e\log_2 x + O(1)}{\lfloor\gamma\log_2 x\rfloor}\bigg)^{\lfloor\gamma\log_2 x\rfloor} \\
&\ll \bigg(\frac{e}{\gamma}\Big(1 + O\Big(\frac{1}{\log_2 x}\Big)\Big)\bigg)^{\lfloor\gamma\log_2 x\rfloor} \ll \Big(\frac{e}{\gamma}\Big)^{\lfloor\gamma\log_2 x\rfloor} \ll_\gamma (\log x)^{\gamma - \gamma\log\gamma}.
\end{align*}
Thus, we have obtained an upper bound of
\[
\ll_\gamma \frac{x(\log_2 x)^5}{(\log x)^{2+\gamma\log\gamma - \gamma}},
\]
as desired.

\section{A lower bound}

\begin{theorem}\label{lowerbound}
Consider $E : y^2 = x^3 - x$ and fix $\gamma > 1$. Then
\[
\#\{p \leq x : \omega(\efp) > \gamma \log_2 x\} \geq \frac{x}{(\log x)^{2 + \gamma\log\gamma - \gamma + o(1)}}.
\]
The same statement is true if instead $0 < \gamma < 1$ and the strict inequality is reversed on the left-hand side.
\end{theorem}

Our strategy in the case $\gamma > 1$ is as follows. As before, we write $\efp = \Vert \pi - 1 \Vert$, where $\pi \equiv 1 \pmod{(1 + i)^3}$ and $p = \pi\ol{\pi}$. Let $k$ be an integer to be specified later and fix an ideal $\mfs \in \bZ[i]$ with the following properties:
\begin{itemize}
\item[(A)] $((1 + i)^3) \mid \mfs$
\item[(B)] $\omega(\mfs) = k$
\item[(C)] $P^{+}(\Vert \mfs \Vert) \leq x^{1/100\gamma\log_2x}$
\item[(D)] Each prime ideal $\mfp \mid \mfs$ (with the exception of $(1 + i)$) lies above a rational prime $p \equiv 1 \pmod 4$
\item[(E)] Distinct $\mfp$ dividing $\mfs$ lie above distinct $p$
\item[(F)] $\mfs$ squarefree
\end{itemize}
Here $P^{+}(n)$ denotes the largest prime factor of $n$. Note that we have $\omega(\mfs) = \omega(\Vert \mfs \Vert)$. First, we will estimate from below the size of the set $\cM_\mfs$, defined to be the set of those $\pi \in \bZ[i]$ with $\Vert \pi \Vert \leq x$ satisfying the following properties:
\begin{enumerate}
	\item $\pi$ prime (in $\bZ[i]$)
	
	\item $\Vert \pi \Vert$ prime (in $\bZ$)
	
	\item $\pi \equiv 1 \pmod \mfs$
	
	\item $P^-\Big(\frac{\Vert\pi - 1\Vert}{\Vert\mfs\Vert}\Big) > x^{1/100\gamma\log_2 x}$.
\end{enumerate}
Here $P^{-}(n)$ denotes the smallest prime factor of $n$. The conditions on the size of the prime factors of $\Vert \mfs \Vert$ and $\Vert \pi - 1 \Vert / \Vert \mfs \Vert$ imply that each $\pi$ with $\Vert \pi \Vert \leq x$ belongs to at most one of the sets $\cM_\mfs$. If $k$ is chosen to be greater than $\gamma \log_2 x$, then carefully summing over $\mfs$ satisfying the conditions above yields a lower bound on the count of distinct $\pi$ corresponding to $p$ with the property that $\omega(\efp) \geq k > \gamma\log_2 x$. The problem of counting elements $\pi$ and $\ol{\pi}$ with $p = \pi\ol{\pi}$ is remedied by inserting a factor of $\frac{1}{2}$, which is of no concern for us.

More care is required in the case $0 < \gamma < 1$, which is handled in Section \ref{smallgamma}.

\subsection{Preparing for the proof of Theorem \ref{lowerbound}} Suppose the fixed ideal $\mfs$ is generated by $\sigma \in \bZ[i]$. We will estimate from below the size of $\cM_\mfs$ using Theorem 2.1. Define $\cA$ to be the sequence of elements of $\bZ[i]$ of the form
\[
\Big\{\frac{\pi - 1}{\sigma} : \Vert \pi \Vert \leq x, \pi \text{ prime, and } \pi \equiv 1 \pmod{\sigma}\Big\}.
\] 
Let $\cP$ denote the set of prime ideals $\{ \mfp : \Vert \mfp \Vert \leq z\}$, where $z := x^{1/50\gamma\log_2 x}$. Let $\mfP := \prod_{\mfp \in \cP} \mfp$. If $\frac{\pi-1}{\sigma} \equiv 0 \pmod \mfp$ implies $\Vert \mfp \Vert \geq z$, then all primes $p \mid \Vert \frac{\pi - 1}{\sigma} \Vert$ have $p > x^{1/100\gamma\log_2x}$. Note also that if a prime $\pi \in \bZ[i]$, $\Vert \pi \Vert \leq x$ is such that $\Vert \pi \Vert$ is not prime, then $\Vert \pi \Vert = p^2$ for some rational prime $p$, and so the count of such $\pi$ is clearly $O(\sqrt{x})$. Therefore, we have
\[
\#\cM_\mfs \geq S(\cA, \cP) + O(\sqrt{x}).
\]

\begin{lemma}\label{msigmasieve}
With $\cM_\mfs$ defined as above, we have
\[
\#\cM_\mfs \geq c \cdot \frac{\Li(x)\log_2 x}{\Phi(\mfs)\log x} + O\bigg(\sum_{\substack{\mfu \mid \mfP \\ \omega(\mfu) \leq m}} |r(\mfu\mfs)|\bigg) + O\bigg(\frac{1}{\Phi(\mfs)}\frac{\Li(x)}{(\log x)^{22}}\bigg) + O(\sqrt{x}),
\]
where $r(\mfv) = |\frac{\Li(x)}{\Phi(\mfv)} - \pi(x; \mfv, 1)|$ and $c > 0$ is a constant.
\end{lemma}

\begin{proof}
First, note that we expect the size of $\cA$ to be approximately $X := 4\frac{\Li(x)}{\Phi(\mfs)}$. Write $A_\mfu = \#\{a \in \cA : \mfu \mid a\}$. Then
\[
A_\mfu = X\delta(\mfu) + r(\mfu\mfs),
\]
where $\delta(\mfu) = \frac{\Phi(\mfs)}{\Phi(\mfu\mfs)}$ and $r(\mfu\mfs) = |4\frac{\Li(x)}{\Phi(\mfu\mfs)} - \pi(x; \mfu\mfs, 1)|$. By Theorem \ref{brun}, for any even integer $m \geq 0$ we have
\begin{align*}
S(\cA, \cP) = 4\frac{\Li(x)}{\Phi(\mfs)} \prod_{\Vert \mfp \Vert \leq z}\bigg(1 - &\frac{\Phi(\mfs)}{\Phi(\mfp\mfs)}\bigg) + O\bigg(\sum_{\substack{\mfu \mid \mfP \\ \omega(\mfu) \leq m}} |r(\mfu\mfs)|\bigg) \\ &+ O\bigg(\frac{\Li(x)}{\Phi(\mfs)}\sum_{\substack{\mfu \mid \mfP \\ \omega(\mfu) \geq m}} \delta(\mfu)\bigg).
\end{align*}

Using Proposition \ref{mertens}, we have
\begin{align*}
\prod_{\Vert \mfp \Vert \leq z}\bigg(1 - \frac{\Phi(\mfs)}{\Phi(\mfp\mfs)}\bigg) &= \prod_{\substack{\Vert \mfp \Vert \leq z \\ \mfp \nmid \mfs}} \bigg(1 - \frac{1}{\Phi(\mfp)}\bigg) \prod_{\substack{\Vert \mfp \Vert \leq z \\ \mfp \mid \mfs}} \bigg(1 - \frac{1}{\Vert \mfp \Vert}\bigg) \\
&= \prod_{\Vert \mfp \Vert \leq z}\bigg(1 - \frac{1}{\Vert \mfp \Vert}\bigg)\prod_{\substack{\Vert \mfp \Vert \leq z \\ \mfp \nmid \mfs}}\bigg(1 - \frac{1}{(\Vert \mfp \Vert - 1)^2}\bigg) \\ 
&\gg \frac{1}{\log z} = \frac{\log_2 x}{\log x}.
\end{align*}

Take $m = 14\lfloor \log_2 x \rfloor$. We leave aside the first $O$-term and concentrate for now on the second. This term is handled in essentially the same way as in the proof of the upper bound: The sum in the this term is bounded from above by
\[
\sum_{s \geq m} \frac{1}{s!}\Big(\sum_{\substack{\Vert \mfp \Vert \leq z}} \delta(\mfp)\Big)^s.
\]
By Proposition \ref{mertens}, we have
\[
\sum_{\substack{\Vert \mfp \Vert \leq z}} \delta(\mfp) \leq \log_2 x + O(1).
\]
Now, one sees once again by the ratio test that the sum on $s$ is 
\[
\ll \frac{1}{m!}\Big(\sum_{\substack{\Vert \mfp \Vert \leq z}} \delta(\mfp)\Big)^m \leq \frac{1}{m!}(\log_2 x + O(1))^m.
\]
Thus, by the same calculations as in the proof of Theorem \ref{upperbound}, the second $O$-term is
\[
\ll \frac{\Li(x)}{\Phi(\mfs)(\log x)^{22}},
\]
completing the proof of the lemma.
\end{proof}

We now sum this estimate over $\sigma$ in an appropriate range to deal with the $O$-terms and establish a lower bound. Here, the cases $\gamma > 1$ and $0 < \gamma < 1$ diverge.

\subsection{The case $\gamma > 1$.} The argument in this case is somewhat simpler. Recall that $\mfs$ is chosen to satisfy properties A through F listed below Theorem \ref{lowerbound}; in particular, $\omega(\mfs) = k$ for some integer $k$ and $P^{+}(\Vert \mfs \Vert) \leq x^{1/100\gamma\log_2x}$. Choose $k := \lfloor \gamma\log_2x \rfloor + 2$. Since $\omega(\Vert \mfs \Vert) = \omega(\mfs)$, we have that $\Vert \mfs \Vert \leq x^{k/100\gamma\log_2 x} \leq x^{1/10}$. A lower bound follows by estimating the quantity
\[
\cM = \sideset{}{'}\sum_{\mfs} \#\cM_\mfs,
\]
where the prime indicates a restriction to those ideals $\mfs \subset \bZ[i]$ satisfying properties A through F mentioned above.

\begin{lemma}\label{lowerboundonm}
We have
\[
\cM \gg \frac{x\log_2 x(\log_2 x + O(\log_3 x))^k}{k!(\log x)^2}.
\]
\end{lemma}

\begin{proof}

Since $\sum_{\Vert \mfs \Vert \leq x} 1/\Phi(\mfs) \ll \log x$, the second $O$-term in Lemma \ref{msigmasieve} is, upon summing on $\mfs$, bounded by a constant times $\Li(x)/(\log x)^{21}$. The third error term, $O(\sqrt{x})$, is therefore safely absorbed by this term.

We now handle the sum over $\mfs$ of the first $O$-term. We have $|r(\mfu\mfs)| = |\pi(x; \mfu\mfs, 1) - 4\frac{\Li(x)}{\Phi(\mfu\mfs)}|$. We can think of the double sum (over $\mfs$ and $\mfu$) as a single sum over a modulus $\mfq$, inserting a factor of $\tau(\mfq)$ to account for the number of ways of writing $\mfq$ as a product of two ideals in $\bZ[i]$. (Here, $\tau(\mfq)$ is the number of ideals in $\bZ[i]$ which divide $\mfq$.) Recalling our choice of $m = 14\lfloor \log_2 x \rfloor$, we have
\begin{align*}
\sum_{\Vert \mfs \Vert \leq x^{1/10}} \sum_{\substack{\mfu \mid \mfP \\ \omega(\mfu) \leq m}} |r(\mfu\mfs)| &\ll \sum_{\Vert \mfq \Vert < x^{2/5}} \Big\vert\pi(x; \mfq, 1) - \frac{\Li(x)}{\Phi(\mfq)}\Big\vert \cdot \tau(\mfq).
\end{align*}
The restriction $\Vert \mfq \Vert \leq x^{2/5}$ comes from $\Vert \mfs \Vert \leq x^{1/10}$ and $\Vert \mfu \Vert \leq x^{m/50\gamma\log_2x} \leq x^{.28}$, recalling $m = 14\lfloor \log_2x \rfloor$ and $\gamma > 1$. Now, for all $y > 0$ and nonzero $\mfi \subset \bZ[i]$ we have $\pi(y; \mfi, 1) \ll y/\Vert \mfi \Vert$; indeed, the same inequality is true with $\pi(y; \mfi, 1)$ replaced by the count of all proper ideals $\equiv 1 \pmod \mfi$. Thus
\[
\Big|\pi(x; \mfq, 1) - 4\frac{\Li(x)}{\Phi(\mfq)}\Big|  \ll \frac{x}{\Phi(\mfq)}.
\]
Using this together with the Cauchy-Schwarz inequality and Proposition \ref{bv}, we see that, for any $A > 0$,
\begin{align*}
\sum_{\Vert \mfq \Vert < x^{2/5}} \vert\pi(x; \mfq, 1) - 4\frac{\Li(x)}{\Phi(\mfq)}\vert\tau(\mfq) &\ll \sum_{\Vert \mfq \Vert < x^{2/5}} \vert\pi(x; \mfq, 1) - 4\frac{\Li(x)}{\Phi(\mfq)}\vert^{1/2} \Big(\frac{x}{\Phi(\mfq)}\Big)^{1/2} \tau(\mfq) \\
&\ll \Big(x\sum_{\Vert \mfq \Vert < x^{2/5}} \frac{\tau(\mfq)^2}{\Phi(\mfq)}\Big)^{1/2}\Big(\frac{x}{(\log x)^A}\Big)^{1/2}.
\end{align*}
We can estimate this sum using an Euler product:
\begin{align*}
\sum_{\Vert \mfq \Vert < x^{2/5}} \frac{\tau(\mfq)^2}{\Phi(\mfq)} &\ll \prod_{\Vert \mfp \Vert \leq x^{2/5}} \Big(1 + \frac{4}{\Vert \mfp \Vert}\Big) \\
&\leq \exp\Big\{ \sum_{\Vert \mfp \Vert \leq x^{2/5}} \frac{4}{\Vert \mfp \Vert} \Big\} \ll (\log x)^4.
\end{align*}
Collecting our estimates, we see that the total error is at most $x/(\log x)^{A/2 - 2}$, which is acceptable if $A$ is chosen large enough.

For the main term, we need a lower bound for the sum
\begin{align}\label{phirecip}
\cM = \sideset{}{'}\sum_{\mfs} \frac{1}{\Phi(\mfs)}.
\end{align}
Let $I = (e^{(\log_2 x)^2/k}, x^{1/10k})$. Define a collection of prime ideals $\cP$ such that each $\mfp \in \cP$ lies above a prime $p \equiv 1 \pmod 4$, each prime $p \equiv 1 \pmod 4$ has exactly one prime ideal lying above it in $\cP$, and $\Vert \mfp \Vert \in I$. We apply Lemma \ref{elementary}, with the $y_i$ chosen to be of the form $1/\Phi(\mfp)$ with $\mfp \in \cP$, obtaining
\begin{align}\label{sum1}
\frac{1}{\Phi((1 + i)^3)} &\sideset{}{'}\sum_{\mfs : \mfp \mid (\mfs/(1 + i)^3) \implies \mfp \in \cP}  \frac{1}{\Phi(\mfs/(1 + i)^3)} \\ &\gg \frac{1}{(k-1)!} \Bigg(\sum_{\mfp \in \cP} \frac{1}{\Phi(\mfp)}\Bigg)^{k-1}\Bigg(1 - \binom{k-1}{2} \Big(\frac{1}{S_1^2}\Big)\sum_{\mfp \in \cP} \frac{1}{\Phi(\mfp)^2}\Bigg), \nonumber
\end{align}
where
\[
S_1 = \sum_{\mfp \in \cP} \frac{1}{\Phi(\mfp)}.
\]
By Theorem \ref{mertens}, $S_1 = \tfrac{1}{2}\log_2 x - 2\log_3 x + O(1)$. This introduces a factor of $\frac{1}{2^{k-1}}$ to the right-hand side of (\ref{sum1}), but this is of no concern: If each of the $k$ prime factors of $\mfs$, excluding $(1 + i)$, lies above a distinct prime $p \equiv 1 \pmod 4$, then there are $2^{k-1}$ such ideals $\mfs$ of a given norm. Thus, if we extend the sum on the left-hand side of (\ref{sum1}) to range over all $\mfs$ counted in primed sums (cf. the discussion above Lemma \ref{lowerboundonm}), we obtain
\begin{align*}
\sideset{}{'}\sum_{\mfs} \frac{1}{\Phi(\mfs)} \geq \frac{2^{k-1}}{(k-1)!} \Bigg(\frac{1}{2}\log_2x &- 2\log_3 x + O(1)\Bigg)^{k-1} \\ 
 &\times \Bigg(1 - \binom{k-1}{2} \Big(\frac{1}{S_1^2}\Big)\sum_{\mfp \in \cP} \frac{1}{\Phi(\mfp)^2}\Bigg).
\end{align*}
The quantity $\binom{k-1}{2}$ is bounded from above by $\lceil \gamma\log_2 x \rceil^2$, and the sum on $1/\Phi(\mfp)^2$ tends to 0 as $x \to \infty$. Therefore,
\begin{align*}
1 - \binom{k-1}{2} \Big(\frac{1}{S_1^2}\Big)\sum_{\mfp \in \cP} \frac{1}{\Phi(\mfp)^2} \geq 1 - 4\gamma^2\sum_{\mfp \in \cP} \frac{1}{\Phi(\mfp)^2} \geq \frac{1}{2}
\end{align*}
for large enough $x$, and so
\[
\frac{x\log_2 x}{(\log x)^2}\sideset{}{'}\sum_{\mfs} \frac{1}{\Phi(\mfs)} \gg \frac{x\log_2 x(\log_2 x + O(\log_3 x))^{k-1}}{(k-1)!(\log x)^2},
\]
as desired.
\end{proof}

With $k = \lfloor \gamma\log_2x \rfloor + 2$ and by the more precise version of Stirling's formula $n! \sim \sqrt{2\pi n}(n/e)^n$, we have
\begin{align*}
\frac{(\log_2 x + O(\log_3 x))^{k-1}}{(k-1)!} &\gg \frac{1}{\sqrt{\log_2 x}} \bigg(\frac{e\log_2 x + O(\log_3 x)}{\lfloor\gamma\log_2 x\rfloor}\bigg)^{\lceil\gamma\log_2 x\rceil} \\
&= \frac{1}{\sqrt{\log_2 x}}\bigg(\frac{e}{\gamma}\Big(1 + O\Big(\frac{\log_3x}{\log_2 x}\Big)\Big)\bigg)^{\lceil\gamma\log_2 x\rceil} \\
&= (\log x)^{\gamma - \gamma\log\gamma + o(1)}.
\end{align*}
This yields a main term of the shape
\[
\frac{x}{(\log x)^{2 + \gamma\log\gamma - \gamma+ o(1)}},
\]
which completes the proof of Theorem \ref{lowerbound} in the case $\gamma > 1$.

\subsection{The case $\mathbf{0 < \gamma < 1}$.}\label{smallgamma} Above, we used the fact that if $\pi - 1$ is divisible by certain $\mfs \subset \bZ[i]$ with $\omega(\Vert \mfs \Vert) = k$, then $\Vert \pi - 1 \Vert$ will have at least $k > \gamma\log_2 x$ prime factors. The case $0 < \gamma < 1$ is requires more care: We need to ensure that the quantity $\Vert \pi - 1 \Vert / \Vert \mfs \Vert$ does not have too many prime factors.

\begin{lemma}\label{smallgammasieve}
For any $\mfs \subset \bZ[i]$ satisfying properties A through F listed below Theorem \ref{lowerbound}, we have
\[
\#\{\pi \in \cM_\mfs : \omega\bigg(\frac{\Vert\pi - 1\Vert}{\Vert\mfs\Vert}\bigg) > \frac{\log_2x}{\log_4x}\} \ll \frac{x}{\Vert\mfs\Vert(\log x)^A}.
\]
\end{lemma}

Upon discarding those $\pi$ counted by the above lemma, the remaining $\pi$ will have the property that $\omega(\Vert\pi - 1 \Vert) \in [k, k + \log_2x/\log_4x]$. Choosing $k$ to be the greatest integer strictly less than $\gamma\log_2 x - \log_2x/\log_4x$ ensures that $\Vert \pi - 1 \Vert < \gamma\log_2 x$.

\begin{proof}[Proof of Lemma \ref{smallgammasieve}.]
We begin with the observation that, for any $\mfs \subset \bZ[i]$ under consideration and $\pi \in \cM_\mfs$, we have $\Vert \pi - 1 \Vert/\Vert \mfs \Vert \leq 2x/\Vert \mfs \Vert$. Therefore, we estimate
\[
\sum_{\substack{\Vert \mfa \Vert \leq \frac{2x}{\Vert \mfs \Vert} \\ \omega(\Vert \mfa \Vert) > \log_2x/\log_4x \\ P^{-}(\Vert \mfa \Vert) > x^{1/100\gamma\log_2 x}}} 1 \leq \frac{2x}{\Vert\mfs\Vert} \sum_{\substack{\Vert \mfa \Vert \leq \frac{2x}{\Vert \mfs \Vert} \\ \omega(\Vert \mfa \Vert) > \log_2x/\log_4x \\ P^{-}(\Vert \mfa \Vert) > x^{1/100\gamma\log_2 x}}} \frac{1}{\Vert\mfa\Vert}. 
\]
Noting that $\omega(\Vert \mfa \Vert) \leq \omega(\mfa)$ for any $\mfa \subset \bZ[i]$, by Theorem \ref{mertens} and Stirling's formula, we have
\begin{align*}
\sum_{\substack{\Vert \mfa \Vert \leq \frac{2x}{\Vert \mfs \Vert} \\ \omega(\Vert \mfa \Vert) > \log_2x/\log_4x \\ P^{-}(\Vert \mfa \Vert) > x^{1/100\log_2 x}}} \frac{1}{\Vert\mfa\Vert} &\leq \sum_{\substack{\Vert \mfa \Vert \leq \frac{2x}{\Vert \mfs \Vert} \\ \omega(\mfa) > \log_2x/\log_4x \\ P^{-}(\Vert \mfa \Vert) > x^{1/100\log_2 x}}} \frac{1}{\Vert\mfa\Vert} \\ &\leq \sum_{\ell > \log_2x/\log_4x} \frac{1}{\ell!} \Big (\sum_{x^{1/100\log_2 x} \leq \Vert \mfp \Vert \leq \frac{2x}{\Vert\mfs\Vert}} \sum_{m = 1}^\infty \frac{1}{\Vert \mfp \Vert^m} \Big)^\ell \\
&\ll \sum_{\ell > \log_2x/\log_4x} \Big( \frac{e\log_3x + O(1)}{\ell} \Big)^\ell.
\end{align*}

For each $\ell > \log_2x/\log_4x$, we have $(e\log_3x + O(1))/\ell < 1/2$. Thus
\begin{align*}
\sum_{\ell > \log_2x/\log_4x} \Big( \frac{e\log_3x + O(1)}{\ell} \Big)^\ell &\ll \Big( \frac{e\log_3x + O(1)}{\lfloor \log_2x/\log_4x \rfloor + 1} \Big)^{\lfloor \log_2x/\log_4x \rfloor + 1} \\
&\ll \Big(\frac{1}{(\log_2x)^{1 + o(1)}}\Big)^{\log_2x/\log_4x} \ll e^{-2\log_2x\log_3x/\log_4x}.
\end{align*}
This last expression is smaller than $(\log x)^{-A}$, for any $A > 0$. Therefore, for any fixed $A > 0$,
\[
\#\{\pi \in \cM_\mfs : \omega\bigg( \frac{\Vert \pi - 1 \Vert}{\Vert \mfs \Vert}\bigg) > \frac{\log_2x}{\log_4x}\} \ll \frac{x}{\Vert\mfs\Vert(\log x)^A}. \qedhere
\]
\end{proof}

Write
\[
\cM_\mfs' = \{\pi \in \cM_\mfs : \omega\bigg(\frac{\Vert\pi - 1\Vert}{\Vert\mfs\Vert}\bigg) \leq \frac{\log_2x}{\log_4x}\}.
\]
Lemmas \ref{msigmasieve} and \ref{smallgammasieve} show that $\#\cM_\mfs'$ satisfies
\begin{align*}
\#\cM_\mfs' \geq c \cdot \frac{ x\log_2 x}{\Phi(\mfs)(\log x)^2} &+ O\bigg(\sum_{\substack{\mfu \mid \mfP \\ \omega(\mfu) \leq m}} |r(\mfu\mfs)|\bigg) \\ &+ O\bigg(\frac{1}{\Phi(\mfs)}\frac{\Li(x)}{(\log x)^{22}}\bigg) + O\bigg(\frac{x}{\Vert\mfs\Vert(\log x)^A}\bigg) + O(\sqrt{x}),
\end{align*}
for any $A > 0$. Here, all quantities are defined as in the previous section. Just as before, we sum this quantity over $\mfs \subset \bZ[i]$ satisfying conditions A through F listed below Theorem \ref{lowerbound}. Letting $'$ on a sum indicate a restriction to such $\mfs$, we have, by the same calculations as before,
\[
\cM' \gg \frac{x\log_2 x(\log_2 x + O(\log_3 x))^{k-1}}{(k-1)!(\log x)^2},
\]
where 
\[
\cM' = \sideset{}{'}\sum_{\mfs} \#\cM'_\mfs.
\]
Recall that $k$ is chosen to be the largest integer strictly less than $\gamma\log_2x - \log_2x/\log_4x$; then by Stirling's formula,
\begin{align*}
\frac{(\log_2 x + O(\log_3 x))^{k-1}}{(k-1)!} &\gg \frac{1}{\sqrt{\log_2x}}\Big(\frac{e\log_2x + O(\log_3x)}{k - 1}\Big)^{k-1} \\
&\gg \frac{1}{\sqrt{\log_2x}}\Big(\frac{e}{\gamma}\Big(1 + O\big(\frac{1}{\log_4x}\big)\Big)^{\gamma\log_2 x - \log_2x/\log_4x - 1} \\
&\gg (\log x)^{\gamma\log\gamma - \gamma + o(1)}.
\end{align*}
A final assembly of estimates yields Theorem \ref{lowerbound} in the case $0 < \gamma < 1$.

\bibliographystyle{amsalpha}
\bibliography{refs}

\end{document}